\documentclass[letterpaper,11pt,twoside,keywordsasfootnote,addressatend,noinfoline]{article}
\usepackage{fullpage}
\usepackage[english]{babel}
\usepackage{amssymb}
\usepackage{amsmath}
\usepackage{theorem}
\usepackage{epsfig}
\usepackage{subfigure}
\usepackage{mathrsfs}
\usepackage{color}
\usepackage{imsart}

\newtheorem{theorem}{Theorem}

\newtheorem{lemma}[theorem]{Lemma}

\newenvironment{proof}{\noindent{\bf Proof.}}{\hspace*{2mm}~$\square$}
\newenvironment{proofof}[1]{\noindent{\bf Proof of #1.}}{\hspace*{2mm}~$\square$}

\newcommand{\Z}{\mathbb{Z}}

\newcommand{\C}{\mathscr{C}}
\newcommand{\G}{\mathscr{G}}
\newcommand{\V}{\mathscr{V}}
\newcommand{\E}{\mathscr{E}}
\newcommand{\ep}{\epsilon}
\newcommand{\ind}{\mathbf{1}}
\newcommand{\n}{\hspace*{-5pt}}

\DeclareMathOperator{\card}{card}

\begin{document}

\begin{frontmatter}
\title     {The simple exclusion process on finite \\ connected graphs}
\runtitle  {The simple exclusion process on finite connected graphs}
\author    {Shiba Biswal and Nicolas Lanchier}
\runauthor {Shiba Biswal and Nicolas Lanchier}
\address   {School for Engineering of Matter, \\ Transport and Energy, \\ Arizona State University, \\ Tempe, AZ 85287, USA. \\ sbiswal@asu.edu}
\address   {School of Mathematical and Statistical Sciences, \\ Arizona State University, \\ Tempe, AZ 85287, USA. \\ nicolas.lanchier@asu.edu}

\maketitle

\begin{abstract} \ \
 Consider a system of~$K$ particles moving on the vertex set of a finite connected graph with at most one particle per vertex.
 If there is one, the particle at~$x$ chooses one of the~$\deg (x)$ neighbors of its location uniformly at random at rate~$\rho_x$, and jumps
 to that vertex if and only if it is empty.
 Using standard probability techniques, we identify the set of invariant measures of this process to study the occupation time at each vertex.
 Our main result shows that, though the occupation time at vertex~$x$ increases with~$\deg (x) / \rho_x$, the ratio of the occupation times at
 two different vertices converges \emph{monotonically} to one as the number of particles increases to the number of vertices.
 The occupation times are also computed explicitly for simple examples of finite connected graphs: the star and the path.
\end{abstract}



\end{frontmatter}


\section{Introduction}
\label{sec:intro}
 The simple exclusion process, introduced by Spitzer in~\cite{spitzer_1970}, is one of the most popular interacting particle systems with the voter
 model~\cite{clifford_sudbury_1973, holley_liggett_1975} and the contact process~\cite{harris_1974}.
 These three models can be viewed as spatial stochastic models of diffusion, competition, and invasion, respectively.
 In particular, the simple exclusion process consists of a system of symmetric random walks that move independently on a connected graph except that
 jumps onto already occupied vertices are suppressed (exclusion rule) so that each vertex is occupied by at most one particle. \\
\indent All three models have been studied extensively on infinite lattices, and we refer to~\cite{liggett_1985, liggett_1999} for a review
 of their main properties.
 The voter model and the contact process have also been studied on the torus in~$\Z^d$ (the time to consensus of the voter model on the torus
 is studied in~\cite{cox_1989} and the time to extinction of the contact process on the torus is studied in~\cite{durrett1, durrett2, durrett3})
 as well as various finite deterministic and random graphs.
 In contrast, to the best of our knowledge, there is no work about the simple exclusion process on finite graphs in the probability literature with
 the notable exception of the asymmetric nearest neighbor exclusion process on the finite path (each particle jumps to its immediate left or right
 with different probabilities) introduced in~\cite{liggett_1975} and reviewed in~\cite[chapter~III.3]{liggett_1999}.
 The primary motivation, however, was to study the properties of the stationary distribution on a large path and relate them to those of the
 infinite system. \\
\indent The main objective of this paper is to initiate the study of simple exclusion processes on general finite connected graphs:
 a particle at vertex~$x$ now jumps to a vertex chosen uniformly at random among the~$\deg (x)$ neighbor(s) of~$x$.
 We also assume that the rate at which particles jump depend on their location, with the particle at~$x$ jumping at rate~$\rho_x$.
 This modeling approach is motivated by engineering applications such as robotic swarms where particles represent robots moving on a finite
 graph and avoiding each other.
 An important question in this field is how should we choose the rates~$\rho_x$ so that the system converges to a fixed desired distribution?
 But regardless of its potential applications, our model is natural and of interest mathematically, and as a first step we study the invariant
 measures and prove a monotonicity property for the occupation times at different vertices that holds for all finite connected
 graphs and all choices of the rates~$\rho_x$.

\section{Main results}
\label{sec:results}
 Letting~$\G = (\V,\E)$ be a finite connected graph on~$N$ vertices, with vertex set~$\V$ and edge set~$\E$, the process is a continuous-time Markov
 chain whose state at time~$t$ is a configuration
 $$ \eta_t : \V \to \{0, 1 \} \quad \hbox{where} \quad \eta_t (x) =
    \left\{\hspace*{-3pt} \begin{array}{rl} 0 & \hbox{if vertex~$x$ is empty} \vspace*{2pt} \\
                                            1 & \hbox{if vertex~$x$ is occupied by a particle.} \end{array} \right. $$
 Motivated by potential applications in robotic swarms, we assume that particles may jump at a rate that depends on their location, and
 denote by~$\rho_x$ the rate attached to vertex~$x$.
 To describe the dynamics, for all~$x, y \in \V$, we let~$\tau_{x, y} \,\eta$ be the configuration
 $$ (\tau_{x, y} \,\eta) (z) = \eta (x) \,\ind \{z = y \} + \eta (y) \,\ind \{z = x \} + \eta (z) \,\ind \{z \neq x, y \} $$
 obtained from~$\eta$ by exchanging the states at~$x$ and~$y$.
 Then, for all~$\eta, \xi \in \{0, 1 \}^{\V}$, the process jumps from configuration~$\eta$ to configuration~$\xi$ at rate
 $$ q (\eta, \xi) = \frac{\rho_x}{\deg (x)} \ \ind \{\eta_t (x) = 1 \ \hbox{and} \ \xi = \tau_{x, y} \,\eta \ \hbox{for some} \ (x, y) \in \E \}. $$
 In words, the particle at~$x$, if there is one, chooses one of the neighbors of~$x$ uniformly at random at rate~$\rho_x$, and jumps to this vertex
 if and only if it is empty.
 It will be convenient later to identify each configuration~$\eta$ with the subset of vertices occupied by a particle:
 $$ \eta \equiv \{x \in \V : \eta (x) = 1 \} \subset \V. $$
 This defines a natural bijection between the set of configurations and the subsets of the vertex set, and it will be obvious from
 the context whether~$\eta$ refers to a configuration or a subset. \\
\indent The main objective of this paper is to study the fraction of time each vertex is occupied in the long run.
 We can prove that these limits exist and only depend on the initial configuration through its number of particles, so we will write from now on
\begin{equation}
\label{eq:limits}
  p_K (x) = \lim_{t \to \infty} P (x \in \eta_t \,| \,\card (\eta_0) = K) \quad \hbox{for all} \quad 0 < K \leq N \ \hbox{and} \ x \in \V.
\end{equation}
 To state our results, for all~$0 < K \leq N$ and~$B \subset \V$, we define
 $$ \Lambda_K^+ (B) = \{\eta \in \Lambda_K : B \cap \eta = B \} \quad \hbox{and} \quad
    \Lambda_K^- (B) = \{\eta \in \Lambda_K : B \cap \eta = \varnothing \} $$
 where~$\Lambda_K$ is the set of configurations with~$K$ particles.
 In addition, for each vertex~$z \in \V$, each subset~$\eta \subset \V$, and each collection~$\C$ of subsets of~$\V$, we let
 $$ D (z) = \frac{\deg (z)}{\rho_z}, \quad D (\eta) = \prod_{z \in \eta} \,D (z) \quad \hbox{and} \quad \Sigma (\C) = \sum_{\eta \in \C} \,D (\eta). $$
 Using that the sets~$\Lambda_K$ are closed communication classes as well as reversibility to identify the stationary distribution on each~$\Lambda_K$,
 we prove that the limits in~\eqref{eq:limits} are characterized as follows.
\begin{theorem}[occupation time] --
\label{th:limit}
 For all~$0 < K \leq N$ and~$x \in \V$,
\begin{equation}
\label{eq:OccProb}
  p_K (x) = \frac{\Sigma (\Lambda_K^+ (x))}{\Sigma (\Lambda_K)}.
\end{equation}
 In particular, the limits in~\eqref{eq:limits} are characterized by
\begin{equation}
\label{eq:ratio}
 \frac{p_K (x)}{p_K (y)} = \frac{\Sigma (\Lambda_K^+ (x))}{\Sigma (\Lambda_K^+ (y))} \quad \hbox{and} \quad \sum_{z \in \V} \,p_K (z) = K.
\end{equation}
\end{theorem}
 Even though the right-hand side of~\eqref{eq:OccProb} cannot be simplified in general, some interesting properties can be deduced from this expression
 for arbitrary finite connected graphs.
 It is intuitively clear that, the graph being fixed, the probability~$p_K (x)$ increases with~$K$.
 It can also be proved that this probability increases with~$D (x)$ while a more precise and challenging analysis shows that, though the occupation
 time at~$x$ is smaller than the occupation time at~$y$ when~$D (x) < D (y)$,
 the ratio of the two occupation times converges monotonically to one as the number of particles increases.
\begin{theorem}[monotonicity] --
\label{th:monotonicity}
 For all~$1 < K < N - 1$,
 $$ \frac{D (x)}{D (y)} = \frac{p_1 (x)}{p_1 (y)} < \frac{p_K (x)}{p_K (y)} < \frac{p_{K + 1} (x)}{p_{K + 1} (y)} < \frac{p_N (x)}{p_N (y)} = 1
    \quad \hbox{when} \quad D (x) < D (y). $$
\end{theorem}
 It follows from (the proof of) the theorem that, when all the~$D (x)$ are equal, all the vertices are equally likely to be occupied at equilibrium.
 In particular, assuming for simplicity that the particles always jump at rate~$\rho_x \equiv 1$, all the vertices are occupied with the same
 probability~$K/N$ for the process on finite regular graphs.
 Along these lines, the probabilities in~\eqref{eq:OccProb} can be computed explicitly when~$\rho_x \equiv 1$ and most of the vertices have the
 same degree.
 For instance, for the star graph in which all the vertices have degree one except the center, we have the following result.
\begin{theorem}[star] --
\label{th:star}
 For the star graph with~$N$ vertices and center~0,
 $$ \begin{array}{rcl}
    \displaystyle p_K (0) & \n = \n & \displaystyle \frac{(N - 1) K}{(N - 1) K + (N - K)} \vspace*{8pt} \\
    \displaystyle P_K (x) & \n = \n & \displaystyle \bigg(\frac{K}{N - 1} \bigg) \frac{(N - 1) K - (K - 1)}{(N - 1) K + (N - K)}
    \quad \hbox{for} \ \ x \neq 0. \end{array} $$
\end{theorem}
 For the path graph in which all the vertices have degree two except the two endpoints, the algebra is a little more complicated but we can prove the
 following result.
\begin{theorem}[path] --
\label{th:path}
 For the path graph with vertex set~$\V = \{0, 1, \ldots, N - 1 \}$,
 $$ \begin{array}{rcl}
    \displaystyle p_K (x) & \n = \n & \displaystyle \frac{(2N - K - 1) K}{(K - 1) K + 4 (N - K)(N - 1)} \quad \hbox{for} \ \ x = 0, N - 1 \vspace*{8pt} \\
    \displaystyle p_K (x) & \n = \n & \displaystyle \bigg(\frac{K}{N - 2} \bigg) \frac{(K - 2)(K - 1) + 4 (N - K)(N - 2)}{(K - 1) K + 4 (N - K)(N - 1)}
    \quad \hbox{for} \ \ 0 < x < N - 1. \end{array} $$
\end{theorem}
 Taking the ratio of the probabilities in both theorems~(see~\eqref{eq:star-5} and~\eqref{eq:path-5} below) gives results in agreement
 with Theorem~\ref{th:monotonicity}.
 Figure~\ref{fig:examples} shows~\eqref{eq:OccProb} obtained from numerical simulations of the simple exclusion process on the star, the path and the
 two-dimensional grid with~25 vertices, along with our analytical results for the first two graphs.
\begin{figure}[t]
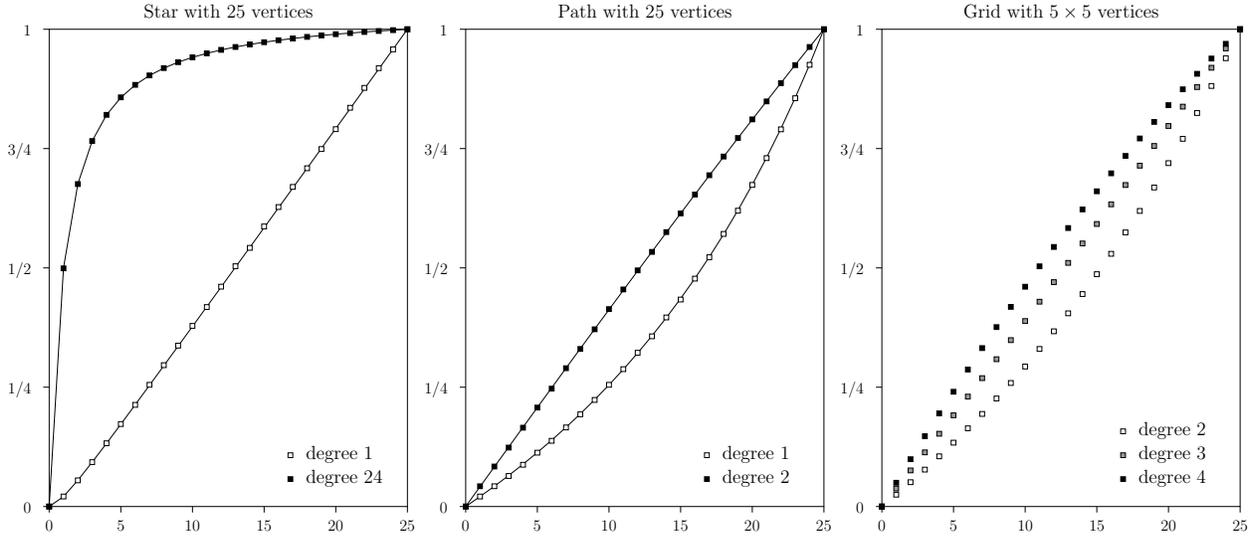

\centering
\scalebox{0.50}{\input{star-graph.pstex_t}}
\scalebox{0.50}{\input{path-graph.pstex_t}}
\scalebox{0.50}{\input{grid-graph.pstex_t}}
\caption{\upshape{Limiting behavior of the simple exclusion process on three specific graphs with~25 vertices and~$\rho_x \equiv 1$.
                  The horizontal axis represents the number of particles~($K$), and the vertical axis represents the fraction of time vertices with
                  degree~$d$ are occupied ($p_K (x)$).
                  The squares are obtained from simulating the process for~$10^8$ units of time, while the curves in the first two pictures show the analytical
                  results from Theorems~\ref{th:star} and~\ref{th:path}.}}
\label{fig:examples}
\end{figure}


\section{Proof of Theorem~\ref{th:limit}}
\label{sec:limit}
 Note that the simple exclusion process is not irreducible because configurations with different numbers of particles do not communicate.
 However, the set of configurations with~$K$ particles forms a closed communication class, so the process restricted to~$\Lambda_K$ is irreducible and
 converges to a unique stationary distribution.
 To find this stationary distribution and prove the theorem, we will also use reversibility.
 To show that~$\Lambda_K$ is a communication class, we first prove that any two configurations in~$\Lambda_K$ that only differ in two vertices communicate.
\begin{lemma} --
\label{lem:x-y}
 For all~$ \eta \in \{0, 1 \}^{\V}$ and~$x, y \in \V$,
 $$ P ( \eta_t = \tau_{x, y} \, \eta \,| \, \eta_0 =  \eta) > 0 \quad \hbox{for all} \quad t > 0. $$
\end{lemma}
\begin{proof}
 The result is obvious when~$ \eta (x) =  \eta (y)$ because in this case~$\tau_{x, y} \, \eta =  \eta$.
 The result is also clear when~$x$ and~$y$ are not in the same state but connected by an edge because
\begin{equation}
\label{eq:x-y-1}
  q (\eta, \tau_{x, y} \,\eta) = \lim_{\ep \downarrow 0} \ \frac{P ( \eta_{t + \ep} = \tau_{x, y} \, \eta \,| \, \eta_t =  \eta)}{\ep}
                               = \frac{\rho_x}{\deg (x)} \quad \hbox{for all} \quad (x, y) \in \E.
\end{equation}
 To deal with the nontrivial case when the vertices are neither in the same state nor connected by an edge, we may assume without loss of generality
 that, in configuration~$ \eta$, vertex~$x$ is occupied and vertex~$y$ empty.
 Because the graph~$\G$ is connected, there exists a self-avoiding path
 $$ (z_1, z_2, \ldots, z_l) \subset \V \quad \hbox{with} \quad (z_i, z_{i + 1}) \in \E, \ z_1 = x \ \hbox{and} \ z_l = y $$
 connecting~$x$ and~$y$.
 Due to the absence of cycles, we have also~$l \leq N$.
 To remove the particle at~$x$ and put a particle at~$y$ without changing the state of the other vertices, we let
 $$ \{i : z_i \in \eta \} = \{z_{i (1)}, z_{i (2)}, \ldots, z_{i (k)} \} \quad \hbox{with} \quad 1 = i (1) < i (2) < \cdots < i (k) < l $$
 be the set of vertices along the self-avoiding path that are occupied in configuration~$ \eta$.
 It is also convenient to set~$i (k + 1) = l$.
 To obtain configuration~$\tau_{x, y} \, \eta$ from~$ \eta$, the basic idea, which is illustrated in Figure~\ref{fig:move}, is to move the particle
 at~$z_{i (k)}$ to~$z_{i (k + 1)}$, then the particle at~$z_{i (k - 1)}$ to~$z_{i (k)}$, and so on.
 To prove that this sequence of events indeed occurs with positive probability, note that, because the
 vertices~$z_{i (k) + 1}, z_{i (k) + 2}, \ldots, z_{i (k + 1)}$ are empty,
 $$  \tau_{z_{i (k)}, z_i (k + 1)} \,\eta =
    (\tau_{z_{i (k + 1) - 1}, z_{i (k + 1)}} \circ \cdots \circ \tau_{z_{i (k) + 1}, z_{i (k) + 2}} \circ \tau_{z_{i (k)}, z_{i (k) + 1}})(\eta). $$
 This, together with~\eqref{eq:x-y-1}, implies that
 $$ P (\eta_t = \tau_{z_{i (k)}, z_{i (k + 1)}} \,\eta \,| \,\eta_0 = \eta) > 0 \quad \hbox{for all} \quad t > 0. $$
\begin{figure}[t]
\centering
\scalebox{0.50}{\input{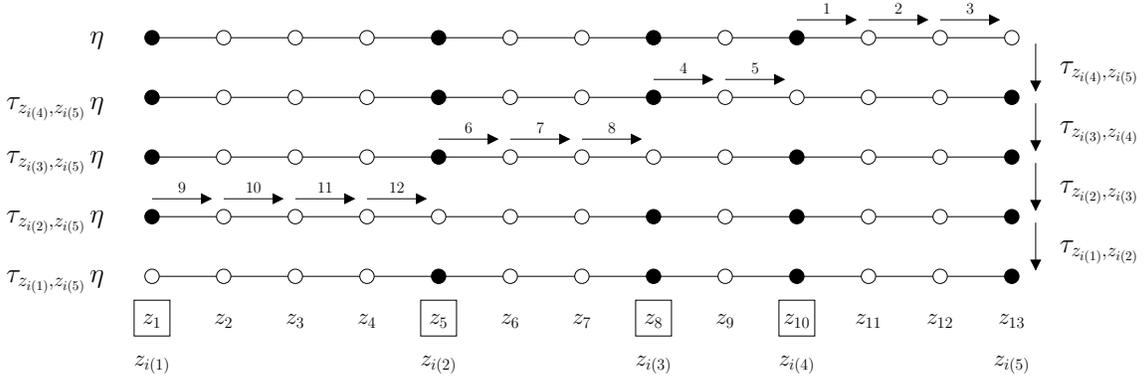}}
\caption{\upshape{Evolution along the self-avoiding path~$(z_1, z_2, \ldots, z_l)$.}}
\label{fig:move}
\end{figure}
 Similarly, we prove by induction that, for~$j = 1, 2, \ldots, k$,
\begin{equation}
\label{eq:x-y-2}
\begin{array}{l}
  P (  \eta_t = \tau_{z_{i (j)}, z_{i (k + 1)}} \, \eta \,| \, \eta_0 = \tau_{z_{i (j + 1)}, z_{i (k + 1)}} \,  \eta) \vspace*{4pt} \\ \hspace*{40pt} =
  P (  \eta_t = \tau_{z_{i (j)}, z_{i (j + 1)}} (\tau_{z_{i (j + 1)}, z_{i (k + 1)}} \, \eta) \,| \,  \eta_0 = \tau_{z_{i (j + 1)}, z_{i (k + 1)}} \, \eta) > 0
\end{array}
\end{equation}
 for all~$t > 0$.
 In addition, for~$j = 1, 2, \ldots, k$,
\begin{equation}
\label{eq:x-y-3}
  \tau_{z_{i (j)}, z_{i (k + 1)}} \,  \eta =
 (\tau_{z_{i (j)}, z_{i (j + 1)}} \circ \cdots \circ \tau_{z_{i (k - 1)}, z_{i (k)}} \circ \tau_{z_{i (k)}, z_{i (k + 1)}})(  \eta).
\end{equation}
 Using~\eqref{eq:x-y-2} and~\eqref{eq:x-y-3}, and that~$z_{i (1)} = x$ and~$z_{i (k + 1)} = y$, we deduce that
 $$ \begin{array}{l}
    \displaystyle  P (  \eta_{kt} = \tau_{x, y} \,  \eta \,| \,  \eta_0 =   \eta) \geq
    \displaystyle \prod_{j = 1}^k \,P (  \eta_{(k - j + 1) t} = \tau_{z_{i (j)}, z_{i (k + 1)}} \,  \eta \,| \,  \eta_{(k - j) t} = \tau_{z_{i (j + 1)}, z_{i (k + 1)}} \,  \eta) \vspace*{-8pt} \\ \hspace*{60pt} =
    \displaystyle \prod_{j = 1}^k \,P (  \eta_t = \tau_{z_{i (j)}, z_{i (k + 1)}} \,  \eta \,| \,  \eta_0 = \tau_{z_{i (j + 1)}, z_{i (k + 1)}} \,  \eta) > 0 \end{array} $$
 for all~$t > 0$.
 This completes the proof.
\end{proof}
\begin{lemma} --
\label{lem:class}
 For all~$K = 0, 1, \ldots, N$, the set~$\Lambda_K$ is a closed communication class.
\end{lemma}
\begin{proof}
 The fact that~$\Lambda_K$ is closed is an immediate consequence of the fact that the number~$K$ of particles is preserved by the dynamics.
 To prove that this set is also a communication class, fix two configurations~$  \eta$ and~$  \xi$ with~$K$ particles, and define the sets
 $$ S = \{z \in \V :   \eta (z) = 1, \,  \xi (z) = 0 \} \quad \hbox{and} \quad T = \{z \in \V :   \eta (z) = 0, \,  \xi (z) = 1 \} $$
 that we call respectively the source set and the target set.
 Because~$  \eta$ and~$  \xi$ have the same number of particles, these sets have the same number of vertices, and we write
 $$ S = \{x_1, x_2, \ldots, x_k \} \subset \V \quad \hbox{and} \quad T = \{y_1, y_2, \ldots, y_k \} \subset \V. $$
 By definition of the source and target sets,
 $$   \xi = (\tau_{x_1, y_1} \circ \tau_{x_2, y_2} \circ \cdots \circ \tau_{x_k, y_k})(  \eta). $$
 In particular, letting~$\sigma_j = \tau_{x_1, y_1} \circ \cdots \circ \tau_{x_j, y_j}$ and using Lemma~\ref{lem:x-y}, we get
 $$ \begin{array}{l}
    \displaystyle P (  \eta_{kt} =   \xi \,| \,  \eta_0 =   \eta) \geq
    \displaystyle \prod_{j = 1}^k P (  \eta_{jt} = \sigma_j \,  \eta \,| \,  \eta_{(j - 1) t} = \sigma_{j - 1} \,  \eta) \vspace*{-8pt} \\ \hspace*{140pt} =
    \displaystyle \prod_{j = 1}^k P (  \eta_t = \tau_{x_j, y_j} (\sigma_{j - 1} \,  \eta) \,| \,  \eta_0 = \sigma_{j - 1} \,  \eta) > 0 \end{array} $$
 for all~$t > 0$.
 Since this holds for any two configurations in~$\Lambda_K$ and since configurations outside~$\Lambda_K$ cannot be reached from~$\Lambda_K$,
 the result follows.
\end{proof} \\ \\
%
 We now use reversibility to identify the stationary distributions.
\begin{lemma} --
\label{lem:reversible}
 For all~$K = 0, 1, \ldots, N$, the distribution
 $$ \pi_K (  \eta) = \frac{D (\eta)}{\Sigma (\Lambda_K)} \quad \hbox{for all} \quad   \eta \in \Lambda_K $$
 is a reversible distribution concentrated on~$\Lambda_K$.
\end{lemma}
\begin{proof}
 Let~$  \eta,   \xi \in \Lambda_K$, $  \eta \neq   \xi$. Then,
 $$ q (  \eta,   \xi) = q (  \xi,   \eta) = 0 \quad \hbox{when} \quad   \xi \neq \tau_{x, y} \,  \eta \quad \hbox{for all} \quad (x, y) \in \E $$
 in which case it is clear that
\begin{equation}
\label{eq:reversible-1}
  \pi_K (  \eta) \,q (  \eta,   \xi) = \pi_K (  \xi) \,q (  \xi,   \eta) = 0.
\end{equation}
 When~$\xi = \tau_{x, y} \,\eta$ for some~$(x, y) \in \E$, with say~$  \eta (x) =   \xi (y) = 1$,
 $$ q (\eta, \xi) = \frac{\rho_x}{\deg (x)} \quad \hbox{and} \quad q (  \xi,   \eta) = \frac{\rho_y}{\deg (y)}. $$
 In this case, because~$\eta \setminus \{x \} = \xi \setminus \{y \}$,
\begin{equation}
\label{eq:reversible-2}
 \begin{array}{rcl}
 \displaystyle \pi_K (  \eta) \,q (  \eta,   \xi) & \n = \n &
 \displaystyle \frac{D (\eta)}{\Sigma (\Lambda_K)} \ q (  \eta,   \xi) =
 \displaystyle \frac{D (\eta \setminus \{x \})}{\Sigma (\Lambda_K)} \bigg(\frac{\deg (x)}{\rho_x} \bigg) q (  \eta,   \xi) \vspace*{8pt} \\ & \n = \n &
 \displaystyle \frac{D (\eta \setminus \{x \})}{\Sigma (\Lambda_K)} =
 \displaystyle \frac{D (\xi \setminus \{y \})}{\Sigma (\Lambda_K)} =
 \displaystyle \pi_K (  \xi) \,q (  \xi,   \eta). \end{array}
\end{equation}
 Combining~\eqref{eq:reversible-1} and~\eqref{eq:reversible-2} gives the result.
\end{proof} \\ \\
\begin{proofof}{Theorem~\ref{th:limit}}
 According to Lemma~\ref{lem:reversible}, $\pi_K$ is a reversible distribution so this is also a stationary distribution.
 See e.g.~\cite[Sec.~10.3]{lanchier_2017} for a proof.
 Now, according to~Lemma~\ref{lem:class}, the set~$\Lambda_K$ is a finite closed communication class for the simple exclusion process therefore there
 is a unique stationary distribution that concentrates on~$\Lambda_K$, and this distribution is the limit of the process starting
 from~$\eta_0 \in \Lambda_K$.
 See e.g.~\cite[Sec.~10.4]{lanchier_2017} for a proof.
 In particular,
 $$ \lim_{t \to \infty} P (\eta_t = \eta \,| \,\eta_0 = \xi) = \pi_K (\eta) = \frac{D (\eta)}{\Sigma (\Lambda_K)} \quad
    \hbox{for all} \quad \eta, \xi \in \Lambda_K $$
 from which it follows that
 $$ p_K (x) = \sum_{\eta \in \Lambda_K : x \in \eta} \pi_K (\eta)
            = \sum_{\eta \in \Lambda_K : x \in \eta} \frac{D (\eta)}{\Sigma (\Lambda_K)}
            = \sum_{\eta \in \Lambda_K^+ (x)} \frac{D (\eta)}{\Sigma (\Lambda_K)}
            = \frac{\Sigma (\Lambda_K^+ (x))}{\Sigma (\Lambda_K)}. $$
 This shows the second part of the theorem, and
 $$ \frac{p_K (x)}{p_K (y)} = \frac{\Sigma (\Lambda_K^+ (x))}{\Sigma (\Lambda_K)} \ \frac{\Sigma (\Lambda_K)}{\Sigma (\Lambda_K^+ (x))}
                            = \frac{\Sigma (\Lambda_K^+ (x))}{\Sigma (\Lambda_K^+ (y))}. $$
 Finally, using that the expected value is linear, we get
 $$ \sum_{z \in \V} \,p_K (z) = \lim_{t \to \infty} \sum_{z \in \V} \,E (\ind \{z \in \eta_t \})
                              = \lim_{t \to \infty} E \bigg(\sum_{z \in \V} \,\ind \{z \in \eta_t \} \bigg) = E (K) = K. $$
 This completes the proof.
\end{proofof}



\section{Proof of Theorem~\ref{th:monotonicity}}
\label{sec:monotonicity}
 In the presence of one or~$N$ particles, we have 
\begin{equation}
\label{eq:monotonicity-0}
  \frac{\Sigma (\Lambda_1^+ (x))}{\Sigma (\Lambda_1^+ (y))} = \frac{D (x)}{D (y)} = \frac{\rho_y \deg (x)}{\rho_x \deg (y)} \quad \hbox{and} \quad
  \frac{\Sigma (\Lambda_N^+ (x))}{\Sigma (\Lambda_N^+ (y))} = \frac{D (\V)}{D (\V)} = 1.
\end{equation}
 In all the other cases, however, the ratios above become much more complicated.
 Also, the theorem cannot be proved using direct calculations.
 The main ingredient is given by the next lemma whose proof relies on a somewhat sophisticated construction.
\begin{lemma} --
\label{lem:injection}
 For all~$0 < K < N$, we have~$\Sigma (\Lambda_{K + 1}) \,\Sigma (\Lambda_{K - 1}) < (\Sigma (\Lambda_K))^2$.
\end{lemma}
\begin{proof}
 The key is to find partitions~$\mathcal P$ of $\Lambda_{K + 1} \times \Lambda_{K - 1}$ and~$\mathcal Q$ of~$\Lambda_K \times \Lambda_K$ such that
\begin{enumerate}
 \item partition~$\mathcal P$ has less elements than partition~$\mathcal Q$, \vspace*{4pt}
 \item each~$A_i \in \mathcal P$ can be paired with a~$B_i \in \mathcal Q$ such that~$\card (A_i) \leq \card (B_i)$, \vspace*{4pt}
 \item for all~$(\eta, \eta') \in A_i$ and~$(\xi, \xi') \in B_i$, we have~$D (\eta) \,D (\eta') = D (\xi) \,D (\xi')$.
\end{enumerate}
 Let $x_1, x_2, \ldots, x_N $ denote the $N$ vertices.
 To construct the partitions, let
 $$ S_{2K} = \{(u_1, u_2, \ldots, u_N) \in \{0, 1, 2 \}^N : u_1 + \cdots + u_N = 2K \} $$
 and~$\phi : \Lambda_{K + 1} \times \Lambda_{K - 1} \to S_{2K}$ and~$\psi : \Lambda_K \times \Lambda_K \to S_{2K}$ defined as
\begin{equation}
\label{eq:injection-0}
  \begin{array}{rclcrcl}
  \phi (\eta, \eta') & \n = \n & u = (u_1, u_2, \ldots, u_N) & \hbox{where} & u_i & \n = \n & \ind \{x_i \in \eta \} + \ind \{x_i \in \eta' \} \vspace*{4pt} \\
  \psi (\xi, \xi') & \n = \n & u = (u_1, u_2, \ldots, u_N) & \hbox{where} & u_i & \n = \n & \ind \{x_i \in \xi \} + \ind \{x_i \in \xi' \}. \end{array}
\end{equation}
 The two functions have the same expression but differ in that they are not defined on the same sets of configurations. The two partitions are then given by
 $$ \begin{array}{rcl}
    \mathcal P & \n = \n & \{\phi^{-1} (u) : u \in S_{2K} \ \hbox{and} \ \phi^{-1} (u) \neq \varnothing \}  \vspace*{4pt} \\
    \mathcal Q & \n = \n & \{\psi^{-1} (u) : v \in S_{2K} \ \hbox{and} \ \psi^{-1} (u) \neq \varnothing \}. \end{array} $$
 See Figure~\ref{fig:injection} for a schematic representation of the two partitions and an illustration of the proof presented below.
 The function~$\psi$ is surjective.
 In contrast, $\phi (\eta, \eta')$ has~$\card (\eta \setminus \eta') \geq 2$ coordinates equal to one therefore it is not surjective:
 $$ S_{2K}^* = \{u \in S_{2K} : \phi^{-1} (u) \neq \varnothing \} \neq S_{2K} $$
 from which it follows that
\begin{equation}
\label{eq:injection-1}
  \card (\mathcal P) = \card (S_{2K}^*) < \card (S_{2K}) = \card (\mathcal Q).
\end{equation}
\begin{figure}[t]
\centering
\scalebox{0.45}{\input{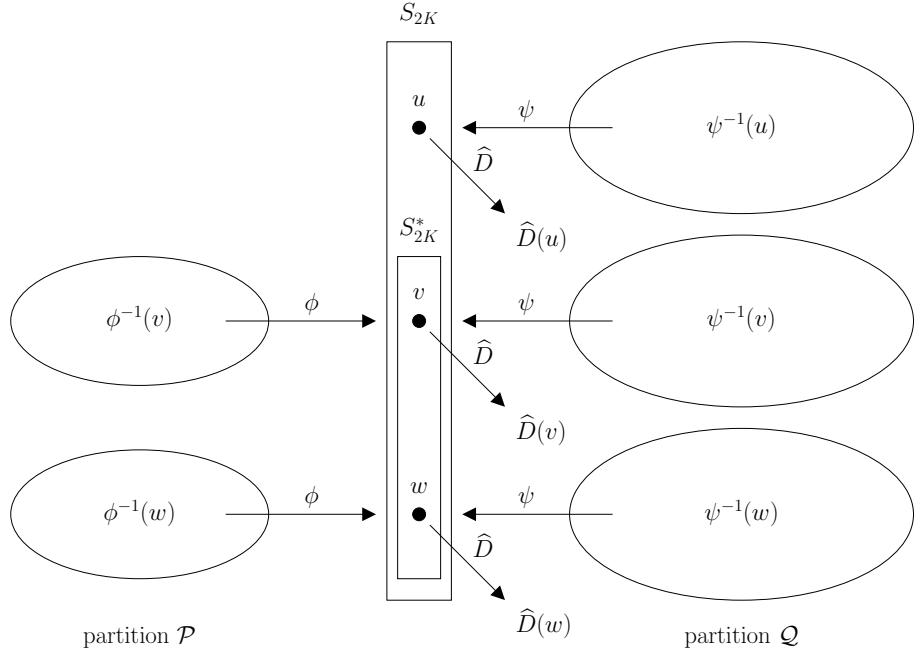}}
\caption{\upshape{Construction in the proof of Lemma~\ref{lem:injection}.}}
\label{fig:injection}
\end{figure}
 This proves the first item above. Now, let
 $$ K_1 = \card \{i : u_i = 1 \} \quad \hbox{and} \quad K_2 = \card \{i : u_i = 2 \}, $$
 with~$K_1 \geq 2$.
 To count the number of preimages~$(\eta, \eta')$ and~$(\xi, \xi')$ in~\eqref{eq:injection-0}, note that the vertices that are either empty
 or occupied in all four configurations are uniquely determined by~$u$.
 This leaves~$K_1$ vertices that are occupied in two of the four configurations and
\begin{itemize}
 \item the number of choices for~$\eta$ is the number of choices of~$K + 1 - K_2$ vertices among~$K_1$ vertices to be occupied in configuration~$\eta$ but not~$\eta'$, \vspace*{4pt}
 \item the number of choices for~$\xi$ is the number of choices of~$K - K_2$ vertices among~$K_1$ vertices to be occupied in configuration~$\xi$ but not~$\xi'$.
\end{itemize}
 Using also that~$K_1 + 2K_2 = 2K < 2K + 1$, we get
\begin{equation}
\label{eq:injection-2}
\begin{array}{rcl}
 \displaystyle \card (\phi^{-1} (u)) & \n = \n &
 \displaystyle {K_1 \choose K + 1 - K_2} = \bigg(\frac{K_1 - K + K_2}{K + 1 - K_2} \bigg) {K_1 \choose K - K_2} \vspace*{8pt} \\ & \n < \n &
 \displaystyle {K_1 \choose K - K_2} = \card (\psi^{-1} (u)), \end{array}
\end{equation}
 which shows the second item above.
 Finally, for all~$(\eta, \eta') \in \phi^{-1} (u)$,
 $$ \begin{array}{rcl}
      D (\eta) \,D (\eta') & \n = & \displaystyle \prod_{z \in \eta} \,\frac{\deg (z)}{\rho_z} \prod_{z \in \eta'} \,\frac{\deg (z)}{\rho_z} =
                                    \displaystyle \prod_{z \in \eta \Delta \eta'} \frac{\deg (z)}{\rho_z}
                                    \displaystyle \prod_{z \in \eta \cap \eta'} \bigg(\frac{\deg (z)}{\rho_z} \bigg)^2 \vspace*{4pt} \\
                           & \n = & \displaystyle \prod_{i = 1}^N \ \bigg(\frac{\deg (x_i)}{\rho_{x_i}} \bigg)^{u_i} = \widehat D (u) \end{array} $$
 is a function~$\widehat D (u)$ of the vector~$u$ only.
 The same holds for~$(\xi, \xi') \in \psi^{-1} (u)$.
 This implies that the third item above is also satisfied in the sense that
\begin{equation}
\label{eq:injection-3}
  D (\eta) \,D (\eta') = D (\xi) \,D (\xi') = \widehat D (u) \quad \hbox{for all} \quad (\eta, \eta') \in \phi^{-1} (u), (\xi, \xi') \in \psi^{-1} (u).
\end{equation}
 Combining~\eqref{eq:injection-1}--\eqref{eq:injection-3}, we conclude that
 $$ \begin{array}{rcl}
    \displaystyle \Sigma (\Lambda_{K + 1}) \,\Sigma (\Lambda_{K - 1}) & \n = \n &
    \displaystyle \sum_{u \in S_{2K}^*} \ \sum_{(\eta, \eta') \in \phi^{-1} (u)} D (\eta) \,D (\eta') \ \overset{\eqref{eq:injection-3}}{=}
    \displaystyle \sum_{u \in S_{2K}^*} \ \sum_{(\eta, \eta') \in \phi^{-1} (u)} \widehat D (u) \vspace*{8pt} \\ & \n \overset{\eqref{eq:injection-2}}{<} \n &
    \displaystyle \sum_{u \in S_{2K}^*} \ \sum_{(\xi, \xi') \in \psi^{-1} (u)} \widehat D (u) \ \overset{\eqref{eq:injection-3}}{=}
    \displaystyle \sum_{u \in S_{2K}^*} \ \sum_{(\xi, \xi') \in \psi^{-1} (u)} D (\xi) \,D (\xi') \vspace*{8pt} \\ & \n \overset{\eqref{eq:injection-1}}{<} \n &
    \displaystyle \sum_{u \in S_{2K}} \ \sum_{(\xi, \xi') \in \psi^{-1} (u)} D (\xi) \,D (\xi') =
    \displaystyle (\Sigma (\Lambda_K))^2. \end{array} $$
 This completes the proof.
\end{proof} \\ \\
 With the previous technical lemma, we can now prove the theorem. \\ \\
\begin{proofof}{Theorem~\ref{th:monotonicity}}
 To simplify the notations, we write
 $$ A_K = \Sigma (\Lambda_K^+ (\{x, y \})) \quad \hbox{and} \quad B_K = \Sigma (\Lambda_{K - 1}^- (\{x, y \})). $$
 Then, we can rewrite
 $$ \begin{array}{rcl}
    \Sigma (\Lambda_K^+ (x)) & \n = \n &
    \Sigma (\Lambda_K^+ (x) \cap \Lambda_K^+ (y)) + \Sigma (\Lambda_K^+ (x) \setminus \Lambda_K^+ (y)) \vspace*{4pt} \\ \hspace*{80pt} & \n = \n &
    \Sigma (\Lambda_K^+ (\{x, y \})) + D (x) \,\Sigma (\Lambda_{K - 1}^- (\{x, y \})) = A_K + D (x) \,B_K. \end{array} $$
 Using some obvious symmetry, we deduce that
\begin{equation}
\label{eq:monotonicity-1}
  \frac{p_K (x)}{p_K (y)} = \frac{\Sigma (\Lambda_K^+ (x))}{\Sigma (\Lambda_K^+ (y))} = \frac{A_K + D (x) \,B_K}{A_K + D (y) \,B_K}.
\end{equation}
 Now, applying Lemma~\ref{lem:injection} to the configurations on~$\V - \{x, y \}$, we get
 $$ \begin{array}{rcl}
     A_K B_{K + 1} & \n = \n & \Sigma (\Lambda_K^+ (\{x, y \})) \,\Sigma (\Lambda_K^- (\{x, y \})) \vspace*{4pt} \\
                   & \n = \n &  D (\{x, y \}) \,\Sigma (\Lambda_{K - 2}^- (\{x, y \})) \,\Sigma (\Lambda_K^- (\{x, y \})) \vspace*{4pt} \\
                   & \n < \n &  D (\{x, y \}) \,\Sigma (\Lambda_{K - 1}^- (\{x, y \})) \,\Sigma (\Lambda_{K - 1}^- (\{x, y \})) \vspace*{4pt} \\
                   & \n = \n & \Sigma (\Lambda_{K + 1}^+ (\{x, y \})) \,\Sigma (\Lambda_{K - 1}^- (\{x, y \})) = A_{K + 1} B_K. \end{array} $$
 This, together with~$D (x) < D (y)$, implies that
 $$ D (x) \,A_{K + 1} B_K + D (y) \,A_K B_{K + 1} < D (x) \,A_K B_{K + 1} + D (y) \,A_{K + 1} B_K $$
 which is equivalent to
\begin{equation}
\label{eq:monotonicity-2}
  \begin{array}{l}
    (A_K + D (x) \,B_K)(A_{K + 1} + D (y) \,B_{K + 1}) \vspace*{4pt} \\ \hspace*{80pt} < \
    (A_{K + 1} + D (x) \,B_{K + 1})(A_K + D (y) \,B_K). \end{array}
\end{equation}
 Combining~\eqref{eq:monotonicity-1} and~\eqref{eq:monotonicity-2} gives
\begin{equation}
\label{eq:monotonicity-3}
  \frac{P_K (x \in \eta_t)}{P_K (y \in \eta_t)} = \frac{A_K + D (x) \,B_K}{A_K + D (y) \,B_K}
                                                = \frac{A_{K + 1} + D (x) \,B_{K + 1}}{A_{K + 1} + D (y) \,B_{K + 1}}
                                                = \frac{p_{K + 1} (x)}{p_{K + 1} (y)}.
\end{equation}
 The theorem is then a combination of~\eqref{eq:monotonicity-0} and~\eqref{eq:monotonicity-3}.
\end{proofof}

	
\section{Star and path graphs}
\label{sec:star-path}
 We now use Theorem~\ref{th:limit} to find the limiting fraction of time each vertex is occupied in the simple exclusion process on the star and the path
 graphs with~$\rho_x \equiv 1$.
 The reason why an explicit calculation is possible in these cases is because most of the vertices have the same degree. \\ \\
\begin{proofof}{Theorem~\ref{th:star}}
 The center~0 is connected to all the other~$N - 1$ vertices therefore its degree is~$N - 1$.
 The other vertices~$1, 2, \ldots, N - 1$, called leaves, are only connected to the center and therefore have degree one.
 For the center, we compute
\begin{equation}
\label{eq:star-1}
 \Sigma (\Lambda_K^+ (0)) = \deg (0) \ \Sigma (\Lambda_{K - 1}^- (0)) = (N - 1) {N - 1 \choose K - 1},
\end{equation}
 while for all~$x = 1, 2, \ldots, N - 1$, we have
 $$ \begin{array}{rcl}
    \Sigma (\Lambda_K^+ (x)) & \n = \n & \deg (x) \ \Sigma (\Lambda_{K - 1}^- (x)) = \Sigma (\Lambda_{K - 1}^- (x)) \vspace*{4pt} \\ & \n = \n &
    \Sigma (\Lambda_{K - 1}^- (x) \cap \Lambda_{K - 1}^+ (0)) + \Sigma (\Lambda_{K - 1}^- (x) \cap \Lambda_{K - 1}^- (0)) \vspace*{4pt} \\ & \n = \n &
    \deg (0) \ \Sigma (\Lambda_{K - 2}^- (0, x)) + \Sigma (\Lambda_{K - 1}^- (0, x)) \end{array} $$
 from which it follows that
\begin{equation}
\label{eq:star-2}
 \Sigma (\Lambda_K^+ (x)) = (N - 1) {N - 2 \choose K - 2} + {N - 2 \choose K - 1}.
\end{equation}
 Note also that the total mass is
\begin{equation}
 \label{eq:star-3}
 \Sigma (\Lambda_K) = \Sigma (\Lambda_K^+ (0)) + \Sigma (\Lambda_K^- (0)) = (N - 1) {N - 1 \choose K - 1} + {N - 1 \choose K}.
\end{equation}
 Combining~\eqref{eq:star-1}--\eqref{eq:star-3}, using the identities
 $$ {N - 1 \choose K - 1} = \frac{N - 1}{K - 1} \ {N - 2 \choose K - 2} = \frac{N - 1}{N - K} \ {N - 2 \choose K - 1} = \frac{K}{N - K} \ {N - 1 \choose K}, $$
 applying Theorem~\ref{th:limit}, and simplifying, we deduce that
\begin{equation}
\label{eq:star-4}
\begin{array}{rcl}
 \displaystyle p_K (0) = \frac{\Sigma (\Lambda_K^+ (0))}{\Sigma (\Lambda_K)} & \n = \n &
 \displaystyle \frac{(N - 1) K}{(N - 1) K + (N - K)} \vspace*{8pt} \\
 \displaystyle p_K (x) = \frac{\Sigma (\Lambda_K^+ (x))}{\Sigma (\Lambda_K)} & \n = \n &
 \displaystyle \bigg(\frac{K}{N - 1} \bigg) \frac{(N - 1) K - (K - 1)}{(N - 1) K + (N - K)}. \end{array}
\end{equation}
 Using again~\eqref{eq:star-1} and~\eqref{eq:star-2}, or alternatively~\eqref{eq:star-4}, we also get
\begin{equation}
\label{eq:star-5}
  \frac{p_K (x)}{p_K (0)} = \frac{\Sigma (\Lambda_K^+ (x))}{\Sigma (\Lambda_K^+ (0))} = \frac{(N - 1)(K - 1) + (N - K)}{(N - 1)^2}.
\end{equation}
 The right-hand side is equal to~$1 / (N - 1)$ when~$K = 1$ and one when~$K = N$, and is increasing with respect to the number of particles, in accordance
 with Theorem~\ref{th:monotonicity}.
 Note also that, for the star graph, the ratio above is linear in the number~$K$ of particles.
\end{proofof} \\ \\
\begin{proofof}{Theorem~\ref{th:path}}
 The end vertices~0 and~$N - 1$ each have degree one while the other vertices~$1, 2, \ldots, N - 2$, each have degree two.
 For the end nodes, we compute
 $$ \begin{array}{rcl}
    \Sigma (\Lambda_K^+ (0)) & \n = \n & \deg (0) \ \Sigma (\Lambda_{K - 1}^- (0)) = \Sigma (\Lambda_{K - 1}^- (0)) \vspace*{4pt} \\ & \n = \n &
    \Sigma (\Lambda_{K - 1}^- (0) \cap \Lambda_{K - 1}^+ (N - 1)) + \Sigma (\Lambda_{K - 1}^- (0) \cap \Lambda_{K - 1}^- (N - 1)) \vspace*{4pt} \\ & \n = \n &
    \deg (N - 1) \ \Sigma (\Lambda_{K - 2}^- (0, N - 1)) + \Sigma (\Lambda_{K - 1}^- (0, N - 1)) \end{array} $$
 from which it follows that
\begin{equation}
 \label{eq:path-1}
 \Sigma (\Lambda_K^+ (0)) = \Sigma (\Lambda_K^+ (N - 1)) = 2^{K - 2} {N - 2 \choose K - 2} + 2^{K - 1} {N - 2 \choose K - 1}.
\end{equation}
 Similarly, for all~$0 < x < N - 1$,
 $$ \Sigma (\Lambda_K^+ (x)) = \deg (x) \ \Sigma (\Lambda_{K - 1}^- (x)) = 2 \,\Sigma (\Lambda_{K - 1}^- (x)), $$
 and including and/or excluding~0 and/or~$N - 1$, we get
\begin{equation}
\label{eq:path-2}
\begin{array}{rcl}
 \Sigma (\Lambda_K^+ (x)) & \n = \n &
 \displaystyle 2^{K - 2} {2 \choose 2}{N - 3 \choose K - 3} + 2^{K - 1} {2 \choose 1}{N - 3 \choose K - 2} + 2^K {2 \choose 0}{N - 3 \choose K - 1} \vspace*{8pt} \\ & \n = \n &
 \displaystyle 2^{K - 2} {N - 3 \choose K - 3} + 2^K {N - 3 \choose K - 2} + 2^K {N - 3 \choose K - 1} \vspace*{8pt} \\ & \n = \n &
 \displaystyle 2^{K - 2} {N - 3 \choose K - 3} + 2^K {N - 2 \choose K - 1}. \end{array}
\end{equation}
 The total mass in this case is
\begin{equation}
 \label{eq:path-3}
 \begin{array}{rcl}
  \Sigma (\Lambda_K) & \n = \n &
  \displaystyle 2^{K - 2} {2 \choose 2}{N - 2 \choose K - 2} + 2^{K - 1} {2 \choose 1}{N - 2 \choose K - 1} +
                2^K {2 \choose 0}{N - 2 \choose K} \vspace*{8pt} \\ & \n = \n &
  \displaystyle 2^{K - 2} {N - 2 \choose K - 2} + 2^K {N - 1 \choose K}. \end{array}
\end{equation}
 Combining~\eqref{eq:path-1}--\eqref{eq:path-3}, using the identities
 $$ \begin{array}{rcl}
    \displaystyle {N - 2 \choose K - 1} & \n = \n &
    \displaystyle \frac{N - K}{K - 1} \ {N - 2 \choose K - 2} \vspace*{8pt} \\ & \n = \n &
    \displaystyle \frac{(N - 2)(N - K)}{(K - 2)(K - 1)} \ {N - 3 \choose K - 3} = \frac{K}{N - 1} \ {N - 1 \choose K}, \end{array} $$
 applying Theorem~\ref{th:limit}, and simplifying, we deduce that
\begin{equation}
\label{eq:path-4}
 \begin{array}{rcl}
  \displaystyle p_K (0) = \frac{\Sigma (\Lambda_K^+ (0))}{\Sigma (\Lambda_K)} & \n = \n &
  \displaystyle \frac{(2N - K - 1) K}{(K - 1) K + 4 (N - K)(N - 1)} \vspace*{8pt} \\
  \displaystyle p_K (x) = \frac{\Sigma (\Lambda_K^+ (x))}{\Sigma (\Lambda_K)} & \n = \n &
  \displaystyle \bigg(\frac{K}{N - 2} \bigg) \frac{(K - 2)(K - 1) + 4 (N - K)(N - 2)}{(K - 1) K + 4 (N - K)(N - 1)}. \end{array}
\end{equation}
 Using again~\eqref{eq:path-1} and~\eqref{eq:path-2}, or alternatively~\eqref{eq:path-4}, we also get
\begin{equation}
\label{eq:path-5}
  \frac{p_K (0)}{p_K (x)} = \frac{\Sigma (\Lambda_K^+ (0))}{\Sigma (\Lambda_K^+ (x))} = \frac{(N - 2)(2N - K - 1)}{(K - 2)(K - 1) + 4 (N - K)(N - 2)}.
\end{equation}
 The right-hand side is equal to one-half when~$K = 1$ and one when~$K = N$, and is increasing with respect to the number of particles, in accordance
 with Theorem~\ref{th:monotonicity}.
\end{proofof}


\bibliographystyle{plain}
\bibliography{biblio.bib}

\end{document}